\documentclass[a4paper,12pt]{amsart}

\usepackage{fullpage}

\makeindex

\usepackage{xcolor}
\usepackage[
colorlinks=true,
citecolor=blue,
urlcolor=blue,
linkcolor=blue,
pdfborder={0 0 0},
breaklinks
]{hyperref}

\usepackage{amsfonts,graphics,amsmath,amsthm,amscd,amssymb,amsmath,latexsym,euscript, enumerate,kotex,mathtools}
\usepackage{epsfig,url}
\usepackage{flafter}
\usepackage[all,cmtip,line]{xy}
\usepackage{array}
\usepackage[english]{babel}
\usepackage{overpic}
\usepackage{subfig}
\usepackage{multirow}
\usepackage{tikz-cd}
\usepackage{theoremref}
\usepackage{microtype}
\usepackage{wrapfig}
\usepackage[shortlabels]{enumitem}
\setlist[enumerate, 1]{1\textsuperscript{o}}

\newtheorem{theorem}{Theorem}[section]
\newtheorem{lemma}[theorem]{Lemma}
\newtheorem{proposition}[theorem]{Proposition}
\newtheorem{conjecture}[theorem]{Conjecture}
\newtheorem{question}[theorem]{Question}
\newtheorem{corollary}[theorem]{Corollary}

\theoremstyle{definition} 
\newtheorem{definition}[theorem]{Definition}
\newtheorem{definition-lemma}[theorem]{Definition-Lemma}

\theoremstyle{remark}
\newtheorem{remark}[theorem]{Remark}

\newtheorem{notation}[theorem]{Notation}
\numberwithin{equation}{section}

\newcommand{\R}{\mathbb{R}}
\newcommand{\Z}{\mathbb{Z}}

\newcommand{\Q}{\mathbb{Q}}

\def\P{\mathbb{P}}

\DeclareRobustCommand{\O}{\mathcal{O}}

\DeclareMathOperator{\ord}{ord}

\def\Spec{\operatorname{Spec}}

\newcommand{\floor}[1]{\left\lfloor #1 \right\rfloor}

\let\oldframe\frame
\renewcommand\frame[1][allowframebreaks]{\oldframe[#1]}

\title[MMP on the generic fiber of log Calabi--Yau type fibration]{Minimal Model Program on the generic fiber \\of log Calabi--Yau type fibration}

\date{\today}

\begin{document}

\author[D.~Kim]{Donghyeon Kim}
\author[D.-W. Lee]{Dae-Won Lee}
\address[Donghyeon Kim]{Department of Mathematics, Yonsei University, 50 Yonsei-ro, Seodaemun-gu, Seoul 03722, Republic of Korea}
\email{narimial0@gmail.com}
\address[Dae-Won Lee]{Department of Mathematics, Ewha Womans University, 52 Ewhayeodae-gil, Seodaemun-gu, Seoul 03760, Republic of Korea}
\email{daewonlee@ewha.ac.kr}

\begin{abstract}
We study the minimal model program on the geometric generic fiber of a fibration $f:X\to S$ such that for a Zariski dense subset $S'\subseteq S$, $X_s$ is an $\varepsilon$-lc log Calabi--Yau type for every $s\in S'$. We prove that for a fibration $f:X\to S$ of varieties, if the fibers are of $\varepsilon$-lc log Calabi--Yau type, then the geometric generic fiber $X_{\overline{\eta}}$ is pklt. In particular, for any big divisor $D$ on $X_{\overline{\eta}}$, we can run the anticanonical MMP and $D$-MMP with scaling of an ample divisor on $X_{\overline{\eta}}$.
\end{abstract}

\maketitle
\allowdisplaybreaks

\section{Introduction}

The minimal model program (MMP) provides a framework for studying the birational geometry of higher dimensional varieties. Controlling the singularities of fibers in a family is an important problem in birational geometry, so that one can run the MMP uniformly and understand the variation of Fano and Calabi--Yau type varieties. In this paper, we focus on fibrations whose closed fibers are of (log) Calabi--Yau type over a Zariski dense subset of the base, and we study the extent to which the geometric generic fiber inherits strong MMP properties from the closed fibers.

\smallskip

Let us briefly recall the definition of log Calabi--Yau type. Let $(X,\Delta)$ be a projective pair and $\varepsilon>0$. We say that $(X,\Delta)$ is of klt (resp. $\varepsilon$-lc)\emph{ log Calabi--Yau type} if there exists an effective $\Q$-Weil divisor $\Delta'\ge 0$ on $X$ such that
$$
(X,\Delta+\Delta') \text{ is klt } (\text{resp. }\varepsilon\text{-lc) and } K_X+\Delta+\Delta' \sim_{\Q} 0.
$$

The Calabi--Yau type pairs naturally arise in birational geometry and behave nicely from the viewpoint of the MMP. Indeed, if $(X,\Delta)$ is a klt log Calabi--Yau type pair, then by definition, there exists an effective $\Q$-divisor $\Delta'$ on $X$ such that $K_X+\Delta+\Delta'\sim_{\Q} 0$. For any effective divisor $D$ on $X$, the pair $(X,\Delta+\Delta'+\varepsilon D)$ is also klt for sufficiently small $\varepsilon>0$. Thus, one can run a $D$-MMP with scaling of an ample divisor. Moreover, if $D$ is big, then this MMP terminates.

\smallskip

Motivated by this observation, we formulate the following question regarding the behavior of log Calabi–Yau type singularities in families.

\begin{question}\label{que}
Let $\varepsilon>0$, $f\colon X\to S$ be a fibration of varieties, and $(X,\Delta)$ a pair. Suppose that there exists a Zariski dense subset $S'\subseteq S$ such that for each $s\in S'$, the fiber $(X_s,\Delta_s)$ is of $\varepsilon$-lc log Calabi--Yau type. Is the geometric generic fiber $(X_{\overline{\eta}},\Delta_{\overline{\eta}})$ of klt log Calabi--Yau type?
\end{question}

\smallskip

A complete answer to Question \ref{que} seems to be subtle. Our first main result gives a weaker conclusion: we show that the geometric generic fiber is potentially klt (pklt for short). For the definition of pklt pairs, see Definition \ref{def:pklt}. We note that if $S'\subseteq S$ is a Zariski dense \emph{open} subset and the base field $k$ is algebraically closed and of infinite transcendence degree over the prime field of any characteristic, then by \cite[Lemma 2.1]{Vial13}, Question \ref{que} has an affirmative answer. However, our question concerns a more general setting. 

\begin{theorem}\label{thm:pklt}
Let $\varepsilon>0$, $f:X\to S$ be a fibration of normal varieties, $(X,\Delta)$ a pair, and let $S'\subseteq S$ be a Zariski dense subset such that $(X_s,\Delta_s)$ is of $\varepsilon$-lc log Calabi--Yau type for each $s\in S'$. Then the geometric generic fiber $(X_{\overline{\eta}},\Delta_{\overline{\eta}})$ is pklt.
\end{theorem}

\smallskip

The pklt condition is strong enough to run the anticanonical MMP by \cite{CJKL25} and, more generally, a $D$-MMP with scaling of an ample divisor on the geometric generic fiber. Since every klt log Calabi--Yau type pair is pklt, the conclusion of our theorem still suffices to highlight many features of the MMP on the geometric generic fiber. As a consequence of Theorem \ref{thm:pklt} and Theorem \ref{thm:pklt MMP}, we obtain the following corollary.

\begin{corollary}\label{cor:MMP}
Under the assumptions of Theorem \ref{thm:pklt}, the following holds:
\begin{itemize}
    \item[\emph{(a)}] there exists a $-(K_{X_{\overline{\eta}}}+\Delta_{\overline{\eta}})$-MMP with scaling of an ample divisor, and
    \item[\emph{(b)}] for any big $\Q$-Cartier divisor $D$ on $X_{\overline{\eta}}$, there exists a $D$-MMP with scaling of an ample divisor that terminates.
\end{itemize}
\end{corollary}

We note that determining whether a pair $(X,\Delta)$ is pklt involves computing asymptotic invariants for every prime divisor $E$ over $X$, which is often intractable in practice. Theorem \ref{thm:pklt} provides an effective way to bypass this difficulty. It demonstrates that the log Calabi--Yau structure of fibers over a Zariski dense subset is a sufficient condition for the geometric generic fiber to be pklt. This allows us to deduce the birational properties of the generic fiber, including the existence of the MMP, from the geometry of the fibration.

\smallskip

We recall a conjecture due to Jiao, which predicts that for a klt Calabi--Yau type variety $X$, there exists a positive integer $m$ such that $m(K_X+\Delta)\sim 0$ for some effective divisor $\Delta$ on $X$ under a boundedness assumption.
\begin{conjecture}[{cf. \cite[Section 3]{Jia25b}}] \label{conj2}
There exists a positive integer $m$ with the following property. Let $\varepsilon,d,v>0$, $X$ be a $d$-dimensional $\varepsilon$-lc Calabi--Yau type variety, and $A$ an ample divisor on $X$ such that $A^d\le v$. Then there exists an effective $\Q$-Weil divisor $\Delta$ on $X$ such that $(X,\Delta)$ is klt and $m(K_X+\Delta)\sim 0$.
\end{conjecture}

Assuming Conjecture \ref{conj2}, we can show that Question \ref{que} has an affirmative answer for $\Delta=0$ case.

\begin{theorem} \label{thm:absolute}
If we assume Conjecture \ref{conj2}, then Question \ref{que} is true for $\Delta=0$.
\end{theorem}

Theorem \ref{thm:pklt} provides evidence for a positive answer to Question \ref{que} in full generality.

\smallskip

The rest of this paper is organized as follows. In Section \ref{Sect:2}, we recall basic definitions and properties on Nakayama's asymptotic order and potentially klt pairs, which will be used in the proof of our main theorems. In Section \ref{Sect:3}, we prove Theorem \ref{thm:pklt}, deduce Corollary \ref{cor:MMP}, and give a proof of Theorem \ref{thm:absolute}.

\section*{Acknowledgement}
The authors are grateful to Minzhe Zhu and Sung Rak Choi for their valuable comments. The authors thank Burt Totaro for his comments and for informing us of the paper \cite{Vial13}. The authors are partially supported by Samsung Science and Technology Foundation under Project Number SSTF-BA2302-03. D.-W. Lee is also partially supported by Basic Science Research Program through the National Research Foundation of Korea (NRF), funded by the Ministry of Education (No. RS-2023-00237440 and 2021R1A6A1A10039823). 

\section{Preliminaries}\label{Sect:2}

Throughout this paper, we work over an algebraically closed field $k$ of characteristic $0$, unless otherwise stated. By a \emph{variety}, we mean an integral, separated scheme of finite type over $k$. All varieties are assumed to be quasi-projective.

\smallskip

We now fix some notation and conventions that will be used throughout the paper. 

\begin{itemize}
    \item[(1)]
    A \emph{couple} $(X,\Delta)$ consists of a normal variety $X$ and an effective $\Q$-Weil divisor $\Delta$ on $X$. If, in addition, $K_X+\Delta$ is $\Q$-Cartier, we call $(X,\Delta)$ a \emph{pair}.

    \item[(2)]
    For a scheme $S$ and a point $s\in S$, we denote $k(s)$ by the residue field of $\mathcal{O}_{S,s}$ and regard $s$ as the morphism $s:\Spec k(s)\to S$.
    The associated geometric point is $\bar{s}:\Spec\overline{k(s)}\to S$.
    For an integral scheme $X$, denote by $k(X)$ its function field.

    \item[(3)]
    If $\eta$ is the generic point of an integral scheme $S$, we write
    $\bar{\eta}:\Spec\overline{k(\eta)}\to S$ and call
    $X_{\bar{\eta}}\coloneqq X\times_S \bar{\eta}$ the \emph{geometric generic fiber} of $X\to S$.

    \item[(4)]
    For a morphism $f:X\to S$ of schemes and a (geometric) point $s\to S$, set
    $$ X_s\coloneqq X\times_S s.$$
    If $\Delta$ is a divisor on $X$, we write $\Delta_s$ for its pullback to $X_s$ whenever it is defined.

    \item[(5)]
    For a morphism $f:X\to S$ between varieties, we say that $f$ is a \emph{fibration} if it is proper, surjective, and $f_*\mathcal{O}_X=\mathcal{O}_S$.
    
    \item[(6)] Let $(X,\Delta)$ be a pair, $f:X'\to X$ a proper birational morphism between normal varieties, and $E$ a prime divisor on $X'$.
    The \emph{log discrepancy} of $E$ with respect to $(X,\Delta)$ is
    $$
      A_{X,\Delta}(E)\coloneqq \mathrm{mult}_E\bigl(K_{X'}-f^*(K_X+\Delta)\bigr)+1.
    $$
    This definition does not depend on the choice of $f$. We say that $(X,\Delta)$ is \emph{kawamata log terminal (klt)} if $A_{X,\Delta}(E)>0$ for every prime divisor $E$ over $X$.

    \item[(7)]
    Let $X$ be a normal variety and $D$ a Cartier divisor on $X$.
    The base ideal of the complete linear system $|D|$ is the image of the natural map
    $$
      H^0(X,\mathcal{O}_X(D))\otimes \mathcal{O}_X(-D)\longrightarrow \mathcal{O}_X,
    $$
    and is denoted by $\mathfrak{b}(|D|)$.

    \item[(8)]
    For a variety $S$, by \emph{shrinking $S$}, we mean replacing $S$ by a variety $S'$ equipped with a quasi-finite morphism $S'\to S$
    (e.g., an open immersion or a generically finite cover), and replacing all data by their pullbacks to $S'$ while keeping the same notation.
\end{itemize}

\subsection{Asymptotic order}

In this subsection, we recall the notion of \emph{asymptotic order} for divisors.

\smallskip

Let $X$ be a normal projective variety and $D$ a big Cartier divisor. For each integer $m\geq 1$, denote by $\mathfrak{b}_m\coloneqq \mathfrak{b}(|mD|)$ the base ideal of $|mD|$. Then the collection $\mathfrak{b}_{\bullet}\coloneqq \left\{\mathfrak{b}_m\right\}_{m\in \Z_{\ge 0}}$ forms a graded sequence of ideals in $\O_X$, where we set $\mathfrak{b}_0\coloneqq \mathcal{O}_X$. For a prime divisor $E$ over $X$, we define the \emph{asymptotic order} of $D$ with respect to $E$ by
$$ \mathrm{ord}_E(\|D\|)\coloneqq \inf_{\substack{m\in \Z_{\geq 1} \\ \mathfrak{b}_m\ne (0)}}\frac{\mathrm{ord}_E(\mathfrak{b}_{m})}{m}.$$
If $D$ is a big $\Q$-divisor on $X$, then we define
$$ \ord_E(\|D\|)\coloneqq \frac{1}{m}\ord_E(\|mD\|),$$
where $n$ is a positive integer such that $mD$ is Cartier. More generally, if $D$ is a pseudoeffective $\Q$-divisor on $X$, then for an ample Cartier divisor $A$ on $X$, we define
$$ \sigma_{E}(D)\coloneqq \lim_{\substack{\varepsilon\to 0^{+} \\ \varepsilon \text{ is rational }}}\mathrm{ord}_E(\|D+\varepsilon A\|).$$
It is known that $\sigma_{E}(D)<\infty$ and the definition does not depend on the choice of $A$ (cf. \cite[Chapter 3.1]{Nak04}). Moreover, if $D$ is big, then $\sigma_{E}(D)=\ord_E(\|D\|)$.

\smallskip

We also note that for any pseudoeffective $\Q$-divisors $D$ and $D'$ on $X$, and any prime divisor $E$ over $X$, we have the following inequality
$$ \sigma_{E}(D+D')\le \sigma_{E}(D)+\sigma_{E}(D').$$

\smallskip

Let $(X,\Delta)$ be a projective klt pair, and $D$ a pseudoeffective $\Q$-divisor on $X$. We define the \emph{log canonical threshold} as
$$ \mathrm{lct}_{\sigma}(X,\Delta,D)\coloneqq \inf_{E}\frac{A_{X,\Delta}(E)}{\sigma_{E}(D)},$$
where the $\inf$ is taken over all prime divisors $E$ over $X$.

Let us recall the following lemma, which is essentially \cite[Proposition 2.1]{FKL16} and appears in the following form in \cite[Lemma 2.11]{Kim25}.

\begin{lemma}[{cf. \cite[Proposition 2.1]{FKL16} and \cite[Lemma 2.11]{Kim25}}] \label{volume asymptotic order}
Let $X$ be a normal projective variety, $D$ a big $\Q$-Cartier $\Q$-divisor, and $E$ a prime divisor on $X$. Suppose that $a$ is a positive rational number. Then $\sigma_E(D)<a$ if and only if
$$ \mathrm{vol}(D-aE)<\mathrm{vol}(D).$$
\end{lemma}

\subsection{Potentially klt pair}
In this subsection, we introduce the notion of \emph{potentially klt} pairs and some of their properties; see, for example, \cite{CP16}, \cite{CJK23}, and \cite{CJKL25} for further details.

\begin{definition} \label{def:pklt}
Let $(X,\Delta)$ be a projective pair. We say that $(X,\Delta)$ is \emph{potentially klt (pklt)} if $-(K_X+\Delta)$ is pseudoeffective, and there is a positive number $\varepsilon>0$ such that
$$ A_{X,\Delta}(E)-\sigma_E(-(K_X+\Delta))>\varepsilon$$
for every prime divisor $E$ over $X$.
\end{definition}

The following is a valuative characterization of pklt pairs, which was recently proved in \cite{CJKL25}.

\begin{theorem}[{cf. \cite[Proposition 4.2]{CJKL25}}] \label{lct}
Let $(X,\Delta)$ be a pklt pair. Then 
$$\mathrm{lct}_{\sigma}(X,\Delta,-(K_X+\Delta))>1.$$
\end{theorem}

The pklt condition is strong enough to guarantee the existence of MMPs.

\begin{theorem} [{cf. \cite[Corollary 1.3]{CJKL25}}] \label{thm:pklt MMP}
Let $(X,\Delta)$ be a pklt pair. Then 
\begin{itemize}
    \item[\emph{(a)}] there exists a $-(K_X+\Delta)$-MMP with scaling of an ample divisor, and
    \item[\emph{(b)}] for every big $\Q$-Cartier divisor $D$ on $X$, there exists a $D$-MMP with scaling of an ample divisor that terminates.
\end{itemize}
\end{theorem}

\begin{proof}
Part (a) is precisely \cite[Corollary 1.3]{CJKL25}.

\smallskip

For (b), we claim that there is a positive number $\varepsilon>0$ such that 
$$\mathrm{lct}_{\sigma}(X,\Delta,-(K_X+\Delta)+\varepsilon D)>1.$$ 
Indeed, we have
$$ 
\begin{aligned}
\frac{1}{\mathrm{lct}_{\sigma}(X,\Delta,-(K_X+\Delta)+\varepsilon D)}&=\sup_{E}\frac{\sigma_{E}(-(K_X+\Delta)+\varepsilon D)}{A_{X,\Delta}(E)}
\\ &\le \sup_{E}\frac{\sigma_{E}(-(K_X+\Delta))+\varepsilon \sigma_{E}(D)}{A_{X,\Delta}(E)}
\\ &\le \sup_{E}\frac{\sigma_{E}(-(K_X+\Delta))}{A_{X,\Delta}(E)}+\varepsilon \sup_{E}\frac{\sigma_{E}(D)}{A_{X.\Delta}(E)}
\\ &=\frac{1}{\mathrm{lct}_{\sigma}(X,\Delta,-(K_X+\Delta))}+\frac{\varepsilon}{\mathrm{lct}_{\sigma}(X,\Delta,D)}.
\end{aligned}
$$
By Theorem \ref{lct}, we have $\mathrm{lct}_{\sigma}(X,\Delta,-(K_X+\Delta))>1$. By choosing $\varepsilon>0$ sufficiently small, we obtain the claim.

\smallskip

Hence, by \cite[Lemma 1.60]{Xu25}, there exists an effective $\Q$-Weil divisor $\Delta'\sim_{\Q}-(K_X+\Delta)+\varepsilon D$ on $X$ such that $(X,\Delta+\Delta')$ is klt. Consequently, $K_X+\Delta+\Delta'\sim_{\Q}\varepsilon D$ and by \cite[Corollary 1.4.2]{BCHM10}, we may run a $(K_X+\Delta+\Delta')$-MMP, i.e., a $D$-MMP with scaling that terminates.
\end{proof}

We also recall the following result, which shows that pklt pairs with a big anticanonical class are of Fano type.

\begin{theorem}[{cf. \cite[Corollary 3.10]{CJK23}}] \label{thm:FT}
Let $(X,\Delta)$ be a pklt pair with $-(K_X+\Delta)$ big. Then $X$ is a variety of Fano type.
\end{theorem}

\begin{remark}
We expect that for a projective klt pair $(X,\Delta)$, the condition that $(X,\Delta)$ is pklt is equivalent to the existence of the anticanonical minimal model of $(X,\Delta)$. The termination of an anticanonical MMP for pklt pairs would imply this equivalence. However, little is known about the termination of such MMPs.

\smallskip

Theorem \ref{thm:FT} was first stated in \cite[Theorem 5.1]{CP16} and its proof relied on \cite[Proposition 4.4]{CP16}. However, the proofs in both the arXiv and published versions of \cite[Proposition 4.4]{CP16} contain a gap. Only recently, a complete proof has been provided in \cite[Corollary 3.10]{CJK23} by using valuation theory.
For more details on this issue, we refer the reader to \cite[Remark 3.11]{CJK23}.
\end{remark}

\subsection{Models}
In this subsection, we introduce the notion of \emph{a model} over a base, in a form analogous to the notion of a model in reduction mod $p$.

\smallskip

\begin{definition}
Let $S$ be a variety and $\eta\in S$ the generic point. Let $X$ be a variety over $\overline{k(\eta)}$.
\begin{itemize}
    \item[(a)] If $X_S$ is a variety flat over $S$ such that $X_S\times_S \overline{k(\eta)}=X$, then we call $X_S$ a \emph{model} of $X$ over $S$.
    \item[(b)] Let $X_S$ be a model of $X$ over $S$ and $D\subseteq X$ a Weil divisor. If $D_S\subseteq X_S$ is an effective Weil divisor on $X_S$ flat over $S$ with $(D_S\subseteq X_S)\times_S \overline{k(\eta)}=D\subseteq X$, then we call $D_S$ a \emph{model} of $D$ over $S$.
    \item[(c)] Let $f:X'\to X$ be a morphism between two varieties over $\overline{k(\eta)}$. If there exists an $S$-morphism $f_S:X'_S\to X_S$ such that $f_S\times_S \overline{k(\eta)}=f$, $X'_S$, and $X_S$ are flat over $S$, then we say that $f_S$ is a \emph{model} of $f$ over $S$.
    \item[(d)] Let $X_S$ be a model of $X$ over $S$ and $ D=\sum q_i D_i$
    a $\Q$-Weil divisor on $X$ with $q_i\in \Q$ and prime $D_i$. If $D_{iS}$ is a model of $D_i$ over $S$ for each $i$, then we say that
    $$ D_S\coloneqq \sum q_iD_{iS}$$
    is a \emph{model} of $D$ over $S$. Note that if $D$ is $\Q$-Cartier, then we may choose $D_S$ as a $\Q$-Cartier divisor on $X_S$.
    \item[(e)] Let $(X,\Delta)$ be a pair, $X_S$ a model of $X$ over $S$, and $\Delta_S$ is so. Then we call $(X_S,\Delta_S)$ a \emph{model} of $(X,\Delta)$ over $S$.
\end{itemize}
\end{definition}

\begin{remark}
Let $S$ be a variety, and $\eta\in S$ the generic point. Then, for every quasi-projective variety $X$ over $\overline{k(\eta)}$, after shrinking $S$, there is a model $X_S$ of $X$ over $S$. Indeed, the variety $X$ is determined by finitely many equations in $\overline{k(\eta)}[x_1,\cdots,x_{d+1}]$ for some $d$. Hence, there exists a finite extension $k'/k(\eta)$ and a scheme $X_{k'}$ over $k'$ such that $X_{k'}\times_{k'}\overline{k(\eta)}=X$. Let $S'\to S$ be a finite morphism, $\eta_{S'}$ the generic point of $S'$ such that $k(\eta_{S'})=k'$, and $X_{S'}$ be the closure of $X_{k'}$ in $\P^d_{S'}$. By generic flatness (cf. \cite[Proposition 052A]{Stacks}), after shrinking $S'$, we may assume that $X_{S'}$ is a flat $S'$-scheme. Replacing $S$ by $S'$ gives a model $X_S$ of $S$. 

\smallskip

A similar argument applies to constructing models of morphisms over $\overline{k(\eta)}$ and a $\Q$-Weil divisor on $X$.
\end{remark}

\begin{notation} \label{nota}
Fix a variety $S$ with a (geometric) point $s$. Let $X_S$ and $X'_S$ be flat $S$-schemes and $f_S:X'_S\to X_S$ be an $S$-morphism. We denote by
$$ X_{s}\coloneqq X_S|_{s}, ~ X'_{s}\coloneqq X'_R|_{s} \text{ and } f_{s}\coloneqq f_S|_{s}.$$
If $D=\sum r_i D_{i}$ is a Weil $\R$-divisor on $X_S$ such that each $D_{i}$ is a prime divisor and flat over $S$, then we denote $D_{s}\coloneqq \sum r_i D_i|_{s}$ as a Weil $\R$-divisor on $X_{s}$. Note that if $D$ is $\Q$-Cartier, then $D_{s}$ is also $\Q$-Cartier.
\end{notation}

We next recall a property of log discrepancy which seems well known to experts.

\begin{proposition}[{cf. \cite[Proposition 2.5]{Kim25}}] \label{disc 1}
Let $(X,\Delta)$ be a pair, $f:X\to S$ a fibration between normal varieties, $X'\to X$ a proper birational morphism, and $E$ a prime divisor on $X'$ which is flat over $S$. Suppose $\eta\in S$ is the generic point. Then after shrinking $S$, we obtain
$$ A_{X_{\overline{\eta}},\Delta_{\overline{\eta}}}(E_{\overline{\eta}})=A_{X_s,\Delta_s}(E_s)$$
for a closed point $s\in S$.
\end{proposition}

\subsection{Volume}

In this subsection, we recall the definition of volume of a divisor and state results on the behavior of volumes in families and on Calabi--Yau type varieties due to Jiao \cite{Jia25a,Jia25b}.

\begin{definition}
Let $X$ be a normal projective variety and $D$ a $\Q$-Cartier divisor on $X$. The \emph{volume} of $D$ is defined as
$$ \mathrm{vol}(D)\coloneqq \limsup_{m\to \infty}\frac{h^0(X,\mathcal{O}_X(\floor{mD}))}{\frac{m^{\dim X}}{(\dim X)!}}.$$
\end{definition}

It is well known that the volume of a divisor has the upper-semicontinuous property.

\begin{theorem}[{cf. \cite[Theorem 1.1]{Jia25a}}] \label{thm:Jiao1}
Let $X\to S$ be a fibration of normal varieties such that $X_{\overline{\eta}}$ is a normal projective variety over $\overline{k(\eta)}$, where $\eta\in S$ is the generic point, and let $D$ be a Cartier divisor on $X$. Then for every Zariski dense subset $S'\subseteq S$, we have $\mathrm{vol}(D_{\overline{\eta}})=\inf_{s\in S'}\mathrm{vol}(D_s).$
\end{theorem}

Moreover, Jiao proved a discreteness property of volumes on $\varepsilon$-lc log Calabi--Yau varieties.

\begin{theorem}[{cf. \cite[Theorem 1.1]{Jia25b}}] \label{thm:Jiao2}
Let $\varepsilon>0$ be a positive number and $d$ a positive integer. Then there exists a discrete set $\mathcal{C}\subseteq \R_{\ge 0}$, depending only on $\varepsilon$ and $d$ with the following property: If $(X,\Delta)$ is a $d$-dimensional $\varepsilon$-lc log Calabi--Yau pair, then for every Cartier divisor $D$ on $X$, we have $\mathrm{vol}(D)\in \mathcal{C}$.
\end{theorem}

\section{Proofs of main results}\label{Sect:3}

\begin{proof}[Proof of Theorem \ref{thm:pklt}]
Let $f:X'\to X_{\overline{\eta}}$ be a proper birational morphism with $E$ a prime divisor on $X'$. Then after shrinking $S$, there exist models $f_S:X'_S\to X$ of $f$ and $E_S$ of $E$ over $S$.

\smallskip

Let $\varepsilon'>0$ be a positive number, and let $A$ be an ample$/S$ Cartier divisor on $X$. Choose an ample$/S$ Cartier divisor $A'$ on $X$ such that $-(K_{X}+\Delta)+\varepsilon' A\le A'$. Then for every $s\in S$, we have $\mathrm{vol}(-(K_{X_s}+\Delta_s)+\varepsilon' A_s)\le \mathrm{vol}(A'_s)=\mathrm{vol}(A'_{\overline{\eta}}).$ Therefore, the set $\left\{\mathrm{vol}(-(K_{X_s}+\Delta_s)+\varepsilon' A_s)\right\}_{s\in S}\subseteq \R_{\ge 0}$ is bounded. Hence, by Theorem \ref{thm:Jiao2}, there exists $S'_{\varepsilon'}\subseteq S'$ that is Zariski dense in $S$ such that $\mathrm{vol}(-(K_{X_s}+\Delta_s)+\varepsilon' A_s)$ is constant for every $s\in S'_{\varepsilon'}$. Moreover, by Theorem \ref{thm:Jiao1}, we have 
\begin{equation} \label{eq:1}
\mathrm{vol}(-(K_{X_s}+\Delta_{s})+\varepsilon' A_s)=\mathrm{vol}(-(K_{X_{\overline{\eta}}}+\Delta_{\overline{\eta}})+\varepsilon' A)\text{ for every }s\in S'_{\varepsilon'}.
\end{equation}
Let $D\coloneqq -(K_X+\Delta)+\varepsilon' A_S$ for simplicity.

\smallskip

Let $a>\sigma_{E_s}(D_s)$ be a rational number. By Lemma \ref{volume asymptotic order}, we have $$\mathrm{vol}(f^*_sD_s-aE_s)<\mathrm{vol}(f^*_sD_s).$$ For any $s\in S'_{\varepsilon}$, we then have $$
\begin{aligned} 
\mathrm{vol}(f^*D_{\overline{\eta}}-aE)&\le \mathrm{vol}(f^*_sD_s-aE_{s}) & (1)
\\ &<\mathrm{vol}(f^*_sD_s) &
\\ &=\mathrm{vol}(D_s) 
\\ &=\mathrm{vol}(D_{\overline{\eta}}) & (2)
\\ &=\mathrm{vol}(f^*D_{\overline{\eta}}),
\end{aligned}$$ 
where we used the upper semicontinuity of cohomology (cf. \cite[Theorem 12.8]{Har77}) in (1), and (\ref{eq:1}) in (2). Therefore, by Lemma \ref{volume asymptotic order} again, we obtain $a>\sigma_E(D)$. Thus,
\begin{equation}\label{eq:2}
\sigma_{E_s}(D_s)\ge \sigma_{E}(D_{\overline{\eta}})\text{ for every }s\in S'_{\varepsilon}.
\end{equation}

\smallskip

By Proposition \ref{disc 1}, after shrinking $S$, we may assume that
\begin{equation} \label{eq:3}
A_{X_{\overline{\eta}},\Delta_{\overline{\eta}}}(E)=A_{X_s,\Delta_s}(E_s)    
\end{equation}
for every $s\in S$. By the assumption that the fiber $(X_s,\Delta_s)$ is of $\varepsilon$-lc log Calabi--Yau type, there exists an effective $\Q$-Weil divisor $\Delta'$ on $X_s$ such that $(X_s,\Delta_s+\Delta')$ is $\varepsilon$-lc and $K_{X_s}+\Delta_s+\Delta'\sim_{\Q}0$.
$$ 
\begin{aligned}
A_{X_{\overline{\eta}},\Delta_{\overline{\eta}}}(E)-\sigma_{E}(D_{\overline{\eta}})&\ge A_{X_s,\Delta_s}(E_s)-\sigma_{E_s}(D_s) &(1)
\\ &= A_{X_s,\Delta_s}(E_s)-\mathrm{ord}_{E_s}(\|-(K_{X_s}+\Delta_s)\|)
\\ &\ge A_{X_s,\Delta_s}(E_s)-\mathrm{ord}_{E_s}\Delta'
\\ &=A_{X_s,\Delta_s+\Delta'}(E_s)>\varepsilon&
\end{aligned}$$
for every $s\in S'_{\varepsilon}$, where (1) is due to (\ref{eq:2}) and (\ref{eq:3}). Hence,
$$ A_{X_{\overline{\eta}},\Delta_{\overline{\eta}}}(E)-\sigma_{E}(-(K_{X_{\overline{\eta}}}+\Delta_{\overline{\eta}}))\ge \varepsilon.$$
Hence, $(X_{\overline{\eta}},\Delta_{\overline{\eta}})$ is pklt, and we complete the proof.
\end{proof}

\begin{remark}
    We emphasize that the $\varepsilon$-lc condition on the fibers is crucial. If we replace this condition with klt or lc, then we do not know whether the volumes of Cartier divisors on a log Calabi--Yau type variety are contained in a discrete set or not. This discreteness is a key ingredient in our proof to ensure that the volume is constant over a Zariski dense subset.
\end{remark}

\begin{proof}[Proof of Corollary \ref{cor:MMP}]
This follows directly from Theorems \ref{thm:pklt} and \ref{thm:pklt MMP}.
\end{proof}

\begin{proof}[Proof of Theorem \ref{thm:absolute}]
The argument follows the strategy of \cite[Step 4 in the proof of Theorem 1.2]{CLZ25}. By Conjecture \ref{conj2}, there exists a positive integer $m$ that is independent of $s\in S'$ such that $m(K_{X}+\Delta)\sim 0$. By \cite[Chapter III, Theorem 12.8 and Corollary 12.9]{Har77}, after replacing $S$ by a finite base change, we obtain a surjection
$$ f_*\mathcal{O}_X(-mK_X)\to H^0(X_s,\mathcal{O}_{X_s}(-mK_{X_s})).$$
Thus, for every $\Delta\in |-mK_{X_s}|$, there exists a divisor $\Delta'\in |-mK_X|$ such that $K_X+\frac{1}{m}\Delta'\sim_{\Q,S}0.$ Moreover, the generic fiber $\left(X_{\overline{\eta}},\frac{1}{m}\Delta'_{\overline{\eta}}\right)$ is klt by \cite[Theorem 9.5.19]{Laz04b}. Since $X_s$ is klt log Calabi--Yau type, we conclude the assertion.
\end{proof}

\bibliographystyle{habbvr}
\bibliography{biblio}

\end{document}